\documentclass{amsart}
\newcommand{\RN}[1]{\textup{\uppercase\expandafter{\romannumeral#1}}}
\usepackage{amsmath}
\usepackage{amssymb}
\usepackage{enumerate}
\usepackage[hidelinks]{hyperref}

\usepackage{mathtools}
\usepackage[capitalise]{cleveref}
\numberwithin{equation}{section}

\newcommand\mc{\mathcal}

\newcommand\mb{\mathbb}
\crefname{equation}{}{}
\newtheorem{theorem}{Theorem}[section]
\newtheorem{lemma}[theorem]{Lemma}

\title{A Note on the Strict Transformation of an Effective Cartier Divisor}
\author{Yu Zhao}
\theoremstyle{definition}
\newtheorem{question}[theorem]{Question}

\newtheorem{definition}[theorem]{Definition}

\theoremstyle{remark}
\newtheorem{remark}[theorem]{Remark}
\usepackage{epigraph}
\usepackage{comment}
\usepackage{tikz}
\usepackage{tikz-cd}

\usepackage{enumitem}

\begin{document}
\maketitle

\begin{abstract}
  Let $Y$ be an effective Cartier divisor of a smooth variety $Z$. Let $X_{i}$, $i\in \{1,\cdots,n\}$ be a set of pairwise disjoint smooth subvarieties in $Y$ such that their union contains the singular locus of $Y$. In this paper, we give a sufficient condition such that $Bl_{X}Y$ is smooth, where $X$ is the disjoint union of all $X_{i}$. Moreover, we prove that assuming a dimensional condition, there is an admissible subcategory of $D^{b}(Bl_{X}Y)$ which is a weak categorical crepant resolution of $Y$, in the sense of Kuznetsov \cite{MR2403307}.
\end{abstract}

\setlength{\epigraphwidth}{0.45\textwidth}

\section{Introduction}

\subsection{The main result of this paper}
In this paper, we work over an algebraically closed field $\mb{F}$ of any characteristic.

Let $Z$ be an irreducible smooth variety and $\{X_{i}\}_{i=1,\cdots n}$ be a set of pairwise disjoint smooth closed subvarieties of $Z$. We denote $X:=\cup_{i=1}^{n}X_{i}$. Let $Y$ be an effective Cartier divisor that contains $X$. The purpose of this paper is to understand the following questions:

\begin{question}
  Let $\tilde{Y}$ be the strict transform of $Y$ in $Bl_{X}Z$. When is $\tilde{Y}$ a smooth divisor of $Bl_{X}Z$?
\end{question}

Let $\pi:Bl_{X}Z\to Z$ be the projection morphism. Let $C_{i}$ be the conormal bundle of $X_{i}$ in $Z$. Then the exceptional divisor $E_{X_{i}}Z\cong \mb{P}_{X_{i}}(C_{i})$ and we denote $\pi_{i}:E_{X_{i}}Z\cong \mb{P}_{X_{i}}(C_{i})\to X_{i}$ as the projection morphisms. If $\tilde{Y}$ is smooth, we have
$$\mc{O}(\tilde{Y})\cong \pi^{*}\mc{O}(Y)-\sum_{i=1}^{n}k_{i}E_{X_{i}}Z$$
for some $k_{i}\in \mb{Z}$. Moreover, if $Y$ is normal, by the adjunction formula, we have
  $$K_{\tilde{Y}}\cong \pi_{Y}^{*}K_{Y}+\sum_{i=1}(d_{i}-k_{i}-1)E_{X_{i}}Z|_{\tilde{Y}},$$
  where $d_{i}$ is the codimension of $X_{i}$ in $Z$, and $K_{\tilde{Y}}$ and $K_{Y}$ are the canonical divisors of $\tilde{Y}$ and $Y$ respectively.

\begin{question}
  How do we compute those coefficients $k_{i}$?
\end{question}

Let $\mc{I}_{i}$ be the ideal sheaf $X_{i}$ in $Z$, and $k_{i}$ be the maximal integer such that  $\mc{O}(-Y)\subset \mc{I}_{i}^{k_{i}}$.  By restricting to $X$, it induces canonical morphisms $\mc{O}(-Y)|_{X_{i}}\to Sym^{k_{i}}(C_{i})$ and thus induce global sections
$$\Phi_{i}\in \Gamma(X_{i},Sym^{k_{i}}(C_{i})\otimes \mc{O}(Y_{i})|_{X_{i}}).$$
Over $\mb{P}_{X_{i}}(C_{i})$, we have canonical morphisms $\pi_{i}^{*}Sym^{k}C_{i}\to \mc{O}_{\mb{P}_{X}(C_{i})}(k_{i})$ and thus $\Phi_{i}$ induce global sections
$$s_{i}\in \Gamma(\mb{P}_{X_{i}}(C_{i}),\mc{O}_{\mb{P}_{X_{i}}(C_{i})}(k_{i})\otimes \pi_{i}^{*}\mc{O}(Y)|_{X_{i}})$$
Let $E_{X_{i}}Y$ be the zero locus of $s_{i}$.  Let $pr:\tilde{Y}\to Y$ and $pr_{i}:E_{X_{i}}Y\to X_{i}$ be the respective projection morphisms. In this paper, we will prove that
\begin{theorem}
  \label{thm:main}
  If $Y-X$ is smooth and $E_{X_{i}}Y$ is a smooth effective divisor of $E_{X_{i}}Z$ for each $i$, then $\tilde{Y}$ is smooth and $E_{X_{i}}Y$ is the exceptional divisor of $X_{i}$. Moreover, we have
  $$\mc{O}(\tilde{Y})\cong \pi^{*}\mc{O}(Y)-\sum_{i=1}^{n}k_{i}E_{X_{i}}Z.$$
\end{theorem}

We denote $B_{i}$ as the zero locus of $\Phi_{i}$. Our second theorem is that
\begin{theorem}
  \label{thm:1.4}
  Assume that all $k_{i}$ is $1$. If $Y-X$ is smooth and for each $i$, $B_{i}$ is smooth and $dim B_{i}=2dim X_{i}-dim Z$, then $\tilde{Y}$ is smooth and
  $$\mc{O}(\tilde{Y})\cong \pi^{*}\mc{O}(Y)-\sum_{i=1}^{n}E_{X_{i}}Z.$$
  Moreover, for any $i$ that $d_{i}=2$, we have the exceptional divisor
  $$E_{X_{i}}Y\cong Bl_{B_{i}}X_{i}.$$
  Under this isomorphism, we have
   \begin{align*}
  E_{B_{i}}X_{i}\cong pr_{i}^{*}det(C_{i})\otimes \mc{O}(\tilde{Y}+2E_{X_{i}}Z)|_{E_{X_{i}}Y},
   \end{align*}
   where $E_{B_{i}}X_{i}$ is the exceptional divisor of $B_{i}$ in $X_{i}$.
 \end{theorem}

\subsection{Lefschetz decomposition of the exceptional divisor and weakly categorical crepant resolutions}
\label{sec:1.2}
In \cite{MR2354207}, Kuznetsov introduced the notion of Lefschetz decomposition:

\begin{definition}[Definition 4.1 of \cite{MR2354207}]
  Let $X$ be an algebraic variety with a line bundle $\mc{L}$ on $X$. A Lefschetz decomposition of the derived category $D^{b}(X)$ is a semi-orthogonal decomposition of $D^{b}(X)$ of the form
  \begin{align*}
    D^{b}(X)=<A_{0},A_{1}\otimes \mc{L},\cdots, \mc{A}_{m-1}\otimes \mc{L}^{m-1}> \\
    0\subset A_{m-1}\subset A_{m-2}\subset \cdots \subset A_{1}\subset A_{0}\subset D^{b}(X)
  \end{align*}
  such that $ 0\subset A_{m-1}\subset A_{m-2}\subset \cdots \subset A_{1}\subset A_{0}\subset D^{b}(X)$ is a chain of admissible sub-categories of $D^{b}(X)$. Lefschetz decomposition is called rectangular if $A_{m-1}=\cdots=A_{1}=A_{0}$.
  
  Similarly, a dual Lefschetz decomposition of $D^{b}(X)$ is a semi-orthogonal decomposition of the form
  $$D^{b}(X)=<B_{m-1}\otimes \mc{L}^{1-m},\cdots, B_{1}\otimes \mc{L}^{-1},B_{0}>$$
  where
  $$0\subset B_{m-1}\subset \cdots \subset B_{1}\subset B_{0}\subset D^{b}(X).$$
\end{definition}

From now on we assume the assumptions in \cref{thm:main} hold.
\begin{theorem}
  \label{thm:1.6}   \label{cor:1.7}
  
  If $k_{i}<d_{i}$, the functor
  $$Lpr_{i}^{*}(-)\otimes \mc{O}_{\mb{P}_{X_{i}(C_{i})}}(l):D^{b}(X_{i})\to D^{b}(E_{X_{i}}Y)$$ is fully faithful for any $l\in \mb{Z}$. Moreover, let
  $$A_{1}^{i}=A_{2}^{i}=\cdots =A_{d_{i}-k_{i}-1}^{i}:= Lpr_{i}^{*}(D^{b}(X_{i}))$$
  and $A_{0}^{i}$ be the left orthogonal of
  $$<A_{1}^{i}\otimes \mc{O}_{\mb{P}_{X_{i}(c_{i})}}(1),\cdots A_{d_{i}-k_{i}-1}^{i}\otimes \mc{O}_{\mb{P}_{X_{i}(c_{i})}}(d_{i}-k_{i}-1)>$$
  in $D^{b}(E_{X_{i}}Y)$. Then $<A_{l}^{i}\otimes \mc{O}_{\mb{P}_{X_{i}(C_{i})}}(i)>_{0\leq l\leq d_{i}-k_{i}-1}$ is a Lefschetz decomposition of $D^{b}(E_{X_{i}}Y_{i})$.

  Similarly, let
  $$B_{d_{i}-k_{i}-1}^{i}\cong \cdots \cong B_{1}^{i}:=Lpr_{i}^{*}(D^{b}(X_{i}))$$
  and $B_{0}$ be the right orthogonal of
  $$<B_{d_{i}-k_{i}-1}^{i}\otimes \mc{O}_{\mb{P}_{X_{i}(c_{i})}}(1+k_{i}-d_{i}),\cdots, B_{1}^{i}\otimes \mc{O}_{\mb{P}_{X_{i}(c_{i})}}(-1)>$$
  in $D_{b}(E_{X_{i}}Y_{i})$. Then we have the dual Lefschetz decomposition
  $$D^{b}(E_{X_{i}}Y)=<B_{d_{i}-k_{i}-1}^{i}\otimes \mc{O}_{\mb{P}_{X_{i}(c_{i})}}(1+k_{i}-d_{i}),\cdots, B_{0}>$$
\end{theorem}

In \cite{MR2403307}, Kuznetsov introduced the notion of categorical resolutions of singularities: 

\begin{definition}[Definition 3.1 of \cite{MR2403307}]
  A triangulated category $D$ is regular if it is equivalent to an admissible subcategory of the bounded derived category of a smooth variety.
\end{definition}

Given a triangulated category $D$, we denote $D^{perf}$ as the full subcategory of perfect objects in $D$.

\begin{definition}[Categorical resolution, Definition 3.2 of \cite{MR2403307}]
  A categorical resolution of a triangulated category $D$ is a regular triangulated category $\tilde{D}$ and a pair of functors
  $$\pi_{*}:\tilde{D}\to D, \pi^{*}:D^{perf}\to \tilde{D},$$
  such that
  \begin{enumerate}
  \item $\pi^{*}$ is left adjoint to $\pi_{*}$ on $D^{perf}$, that is
    $$Hom_{\tilde{D}}(\pi^{*}F,G)\cong Hom_{D}(F,\pi^{*}G)$$
    for any $F\in D^{perf}$, $G\in \tilde{D}$,
  \item the natural morphism of functors $id_{D^{perf}}\to \pi_{*}\pi^{*}$ is an isomorphism.
  \end{enumerate}
 \end{definition}

\begin{definition}[Definition 3.4 of \cite{MR2403307}]
  A categorical resolution $(\tilde{D},\pi_{*},\pi^{*})$ of $D$ is weakly crepant if the functor $\pi^{*}$ is right adjoint to $\pi_{*}$ on $D^{perf}$:
  $$Hom_{\tilde{D}}(G,\pi^{*}F)\cong Hom_{D}(\pi_{*}G,F)$$
  for any $F\in D^{perf}, G\in \tilde{D}$.
\end{definition}

\begin{remark}
  In contrast to the weakly crepant categorical resolution, Kuznetsov \cite{MR2403307} also introduced the notion of strong crepant categorical resolution by requiring the relative Serre functor as identity.
\end{remark}

In \cite{MR2403307}, Kuznetsov revealed the relation between the Lefschetz decomposition of the exceptional divisor and weakly (and strongly) crepant categorical resolution of singularities. By applying his theorem (with a very mild generalization), we prove that
\begin{theorem}
  \label{thm:1.11}
  Assuming the setting of \cref{thm:main} and $k_{i}<d_{i}$ for all $i$, let $\tau_{i}:E_{X_{i}}Y\to \tilde{Y}$ be the closed embeddings of exceptional divisors. Then we have a semi-orthogonal decomposition
  $$D^{b}(\tilde{Y})=<R\tau_{i*}(Lpr_{i}^{*}D^{b}(X_{i})\otimes \mc{O}_{\mb{P}_{X_{i}(C_{i})}}(l_{ij})),\tilde{D}>_{1\leq i\leq n, k_{i}-d_{i}<l_{ij}<0},$$
 such that $\tilde{D}$ is a weak crepant categorical resolution of $D^{b}(Y)$.
\end{theorem}

\subsection{Notations}
All algebraic varieties are assumed to be of finite type over $\mb{F}$. For any algebraic variety $X$, we denote $D^{b}(X)$ as the bounded derived category of coherent sheaves on $X$. For a morphism $f:X\to Y$, we denote $\mb{R}f_{*}$ as the derived push-forward functor and $Lf^{*}$ as the derived pull-back functor.

\subsection{Acknowledgements} This paper is inspired by the recent work of the author on the derived blow-ups of quasi-smooth derived schemes \cite{yuzhaoderived}  and Will Donovan on the homological comparison of resolution and smoothing. Will Donovan suggested that it might be interesting to rewrite some part of \cite{yuzhaoderived} in a more classical and elementary language, and the author would acknowledge him for his encouragement.

The first version of this paper was written when the author was visiting HKUST, Tsinghua YMSC, and CUHK, and the author would like to thank those institutions and Wei-Ping Li, Changjian Su, Micheal Mcbreen, Bin Gui, Siqi He, and Xinzhou Guo for their hospitality and many helpful discussions.

The author is supported by World Premier International Research Center Initiative (WPI initiative), MEXT, Japan, and Grant-in-Aid for Scientific Research grant  (No. 22K13889) from JSPS Kakenhi, Japan.
\section{Smoothness of the Strict Transform}

Before proving \cref{thm:main}, we first prove the following lemma:
\begin{lemma}
  \label{lem:2.1}
  Let $X$ be a smooth variety with pairwise disjoint effective smooth divisors $E_{i}$, $1\leq i\leq m$. Let $Y$ be an effective divisor of $X$. If $Y-\cup_{i=1}^{m}E_{i}$ is smooth and $Y\cap E_{i}$ is a smooth effective divisor of $E_{i}$ for each $i$, then $Y$ is also smooth.
\end{lemma}
\begin{proof}
  By induction, we can reduce to the case that $m=1$. As the smoothness property is a local property, and preserved under the etale morphisms, we can assume that $X$ is an open subvariety of the affine space $\mb{A}^{n}=Spec\ \mb{F}[x_{1},x_{2},\cdots, x_{n}]$, such that $E_{1}$ is the zero locus of $x_{n}$ and $Y$ is the zero locus of
  $$f=x_{n}^{l}g_{n}+x_{n}^{l-1}g_{l-1}+\cdots+ g_{0}$$
  where $g_{i}\in k[x_{1},\cdots, x_{n-1}]$. Then for any $p\in E_{1}$ such that $f(p)=0$, we have
  $$\frac{\partial f}{\partial x_{i}}(p)=\frac{\partial g_{0}}{\partial x_{i}}(p)$$
  for any $1\leq i<n$. Hence the singular locus of $f$ in $Y$ is inside the singular locus of $g_{0}$ in $Y\cap E_{1}$.
\end{proof}
\begin{proof}[Proof of \cref{thm:main}]
  As $\mc{O}(-Y)\subset \mc{I}_{i}^{k_{i}}$, by \cite{MR2163383}, $\pi^{*}\mc{O}(Y)-\sum_{i=1}^{n}k_{i}E_{X_{i}}Z$ is also an effective divisor of $X$, which we denote as $\mc{O}(\tilde{Y})$. Moreover, we have $E_{X_{i}}Z\cong \mb{P}_{X_{i}}(C_{i})$ and $E_{X_{i}}Z\cap \tilde{Y}$ is the zero locus of $s_{i}$. If all the zero locus of $s_{i}$ are smooth effective divisors of $E_{X_{i}}Z$ respectively, by \cref{lem:2.1} $\tilde{Y}$ is smooth. Moreover, it is the schematic closure of $Y-X$ in $Bl_{X}Z$, and thus is the strict transform of $Y$. The preimage of $X_{i}$ in $\tilde{Y}$ is the zero locus of $s_{i}$, and hence is the exceptional divisor $E_{X_{i}}Y$.
\end{proof}

To prove \cref{thm:1.4}, we first prove the following lemmas:

\begin{lemma}
  \label{lem:2.2}
  Let $X$ be a smooth variety. Let $V$ be a locally free sheaf over $X$ and $L$ be a line bundle over $X$. Let $\Phi$ be a global section of $V\otimes L$ and let $S$ be the zero locus of $s$. Let $pr_{V}$ be the projection from $\mb{P}_{X}(V)$ to $X$ and $t$ be the composition of morphisms:
  $$\mc{O}\to pr_{V}^{*}(V\otimes L)\xrightarrow{taut} pr_{V}^{*}L\otimes \mc{O}_{\mb{P}_{X}(V)}(1)$$
  which is a global section of $pr_{V}^{*}L\otimes \mc{O}_{\mb{P}_{X}(V)}(1)$. Let $T$ be the zero locus of $t$. Then $S$ is a $dim(X)-rank(V)$ dimensional smooth subvariety of $X$ iff $T$ is an effective smooth divisor of $\mb{P}_{X}(V)$.
\end{lemma}
\begin{proof}
  Actually, it follows from linear Koszul duality \cite{MR2581249} or the singularity category of projectivization of two-term complexes \cite{jiang2018derived}. Here we give elementary proof for the convenience of readers.

  As the smoothness of $S$ and $T$ are both local properties, and stable under the etale base change of $X$, we can assume that $X$ as an open subvariety of $\mb{A}^{n}=Spec\ \mb{F}[x_{1},\cdots x_{n}]$, and $L\cong \mc{O}_{X}$ and $V\cong \mc{O}_{X}^{m}$ for some positive integer $m$. Then $\Phi$ is represented by
  $$f=(f_{1},f_{2},\cdots f_{m})$$
  where  $f_{i}\in k[x_{1},\cdots x_{m}]$. We consider $\mb{A}^{n+m}=Spec\ \mb{F}[x_{1},\cdots, x_{n},y_{1},\cdots ,y_{m}]$ and
  $$g:=\sum_{i=1}^{m}y_{i}f_{i}.$$
  Here is a $\mb{G}_{m}$ action on $\mb{A}^{n+m}$ by
  $$t\cdot (x_{1},\cdots,x_{n},y_{1},\cdots,y_{m})=(x_{1},\cdots,x_{n},ty_{1},\cdots, ty_{m})$$
  and we denote $T'$ as the zero locus of $g$ on $\mb{A}^{n+m}-\mb{A}^{n}\times \{0\}$. Then $T\cong T'/\mb{G}_{m}$.

  Given $p\in \mb{A}^{n}$, we consider the deravative matrix with the $(i,j)$ entry
  $$Jac(f)_{p}=
  \begin{pmatrix}
    \frac{\partial f_{i}}{\partial x_{j}}(p)
  \end{pmatrix}|_{1\leq i \leq m, 1\leq j\leq n}
  $$
  and the critical locus of $f$, which we denote as $crif(f)$, is defined the the locus that $Jac(f)(p)$ has rank less than $m$. Then the singular locus of $S$ is $S\cap crit(f)$. The singular locus of $T'$ is
  $$\{(x_{1},\cdots,x_{n},y_{1},\cdots, y_{m})|x\in S,Jac(f)_{x}(y)=0\}$$
  where $x:=(x_{1},\cdots,x_{n})$ and $y:=(y_{1},\cdots,y_{m})$. As $y\neq \{0\}$, the singular locus of $T'$ (and also $T$) is empty iff the singular locus of $S$ is empty.
\end{proof}

\begin{lemma}
  \label{lem:2.3}
  We assume all the assumptions of \cref{lem:2.2}. Moreover, we assume that $rank (V)=2$ and $S$ is smooth of $dim(X)-2$. Then $T\cong Bl_{S}X$ and the projection morphism of the blow-up factors through the projection morphism $pr_{V}:\mb{P}_{X}(V)\to X$. Moreover, under the isomorphism, we have
  \begin{equation}
    \label{eq:2.1}
  \mc{O}_{\mb{P}_{X}(V)}(1)|_{T}\cong pr_{V}^{*}(det(V)\otimes L)|_{T} \otimes (E_{S}X)^{-1}
  \end{equation}

\end{lemma}
\begin{proof}
  We have the exact sequence
  $$0\to det(V\otimes L)^{-1}\to V^{\vee}\otimes L^{-1}\xrightarrow{\Phi^{\vee}}\mc{O}_{X}\to \mc{O}_{S}\to 0.$$
  We notice that $V^{\vee}\otimes L^{-1}=det(V)^{-1}\otimes V\otimes L^{-1}$. Thus we have the short exact sequence:
  $$0\to L^{-1}\xrightarrow{\Phi\otimes L^{-1}}V\to det(V)\otimes L\otimes \mc{I}_{S}\to 0$$
  where $\mc{I}_{S}$ is the ideal sheaf of $S$. As
  $$Bl_{S}X\cong Proj_{X}\oplus_{i=0}^{\infty}Sym^{i}(\mc{I})\cong Proj_{X}\oplus_{i=0}^{\infty}Sym^{i}(det(V)\otimes L\otimes \mc{I})$$
  we have $T\cong Bl_{S}X$ and \cref{eq:2.1} follows from the shift of line bundles.
\end{proof}
\begin{proof}[Proof of \cref{thm:1.4}]
  The smoothness of $\tilde{Y}$ follows from \cref{thm:main} and \cref{lem:2.2}. When $d_{i}=2$, we have
  $$E_{X_{i}}Z|_{E_{X_{i}}Y}\cong \mc{O}_{\mb{P}_{X_{i}}(C_{i})}(-1)|_{E_{X_{i}}Y}.$$
  On the other hand, by \cref{lem:2.3}, we have
  \begin{align*}
    E_{B_{i}}X_{i}& \cong (\pi_{i}^{*}(det(C_{i})\otimes \mc{O}(Y))\otimes\mc{O}_{\mb{P}_{X_{i}}(C_{i})}(-1))|_{E_{X_{i}}Y_{i}}\\  &\cong (\pi_{i}^{*}det(C_{i})\otimes \mc{O}(\tilde{Y}+2E_{X_{i}}Z)|_{E_{X_{i}Z}})|_{E_{X_{i}}Y} \\
    & \cong pr_{i}^{*}det(C_{i})\otimes \mc{O}(\tilde{Y}+2E_{X_{i}}Z)|_{E_{X_{i}}Y}.
  \end{align*}
\end{proof}

\section{Lefschetz Decompositions and Weak Crepant Categorical Resolutions}

In this section, we prove \cref{thm:1.6} and \cref{thm:1.11}.
\begin{proof}[Proof of \cref{thm:1.6}]
  Given $F,G\in D^{b}(X_{i})$ and $0\leq l<d_{i}-k_{i}-1$, by the projection formula we have
  $$Hom_{E_{X_{i}}Y}(Lpr_{i}^{*}F,Lpr_{i}^{*}G\otimes \mc{O}_{\mb{P}_{X_{i}}(c_{i})}(l))\cong Hom_{X_{i}}(F\otimes Rpr_{i*}\mc{O}_{E_{X_{i}}Y}(-l),G).$$
  Thus we only need to prove that
  \begin{equation}
    \label{eq:2.2}
    Rpr_{i*}\mc{O}_{E_{X_{i}}Y}\cong \mc{O}_{X_{i}},\quad Rpr_{i*}\mc{O}_{E_{X_{i}}Y}(-l)\cong 0
  \end{equation}
  if $0<l\leq d_{i}-k_{i}-1$. We notice that $E_{X_{i}}Y$ is the zero locus of $s_{i}$ in $\mb{P}_{X_{i}}(C_{i})$. Thus regarding $E_{X_{i}}Y$ as a closed subscheme of $\mb{P}_{X_{i}}(C_{i})$, we have the short exact sequence:
  $$0\to \pi_{i}^{*}\mc{O}(-Y)\otimes \mc{O}_{\mb{P}_{X_{i}}(C_{i})}(-k_{i})\to \mc{O}_{\mb{P}_{X_{i}}(C_{i})}\to \mc{O}_{E_{X_{i}}Y}\to 0.$$
  Thus (\ref{eq:2.2}) follows from Serre's theorem that
  $$R\pi_{*}\mc{O}_{\mb{P}_{X_{i}}(C_{i})}\cong \mc{O}_{X_{i}}, \quad R\pi_{*}\mc{O}_{\mb{P}_{X_{i}}(C_{i})}(m)\cong 0$$
  if $-d_{i}<m<0$.
\end{proof}

Before we prove \cref{thm:1.11}, we summarize Kuznetsov's theory in Section 4 of \cite{MR2403307} in the following theorem (notice that we formulate a mild generalization such that there are several pairwise disjoint exceptional divisors):

\begin{theorem}[Proposition 4.1, Theorem 4.4 and Proposition 4.5 of \cite{MR2403307}]
  \label{thm:3.1}
  Let $pr:\tilde{Y}\to Y$ be a resolution of singularities such that $Y$ is Gorenstein. Let $E_{i}$ ($i\in \{1,\cdots,n\}$) be the exceptional divisors of $pr$ which does not intersect pairwisely, $X_{i}:=pr(E_{i})$, $pr_{i}=pr|_{E_{i}}:E_{i}\to X_{i}$ and $\iota_{i}:E_{i}\to \tilde{Y}$ are the closed embeddings of exceptional divisors. Let $K_{\tilde{Y}}\cong pr^{*}K_{Y}+\sum_{i=1}^{n}(m_{i}-1)E_{i}$, where all $m_{i}\in \mb{Z}$, and $\mc{L}_{i}:=-E_{i}|_{E_{i}}$. If for each $i$, there exists a dual Lefschetz decomposition
  $$D^{b}(E_{i})=<B_{m_{i}-1}^{i}\otimes L_{i}^{1-m}, B_{m-1}^{i}\otimes L_{i}^{2-m},\cdots, B_{1}^{i}\otimes L_{i}^{-1},B_{0}^{i}>$$
  such that $Lpr_{i}^{*}D^{perf}(X_{i})\subset B_{m_{i}-1}^{i}$, then there exists a semi-orthogonal decomposition
  \begin{align*}
    D^{b}(\tilde{Y})= &<R\iota_{1*}(B_{m_{1}-1}^{1}\otimes L_{1}^{1-m_{1}}),\cdots, R\iota_{1*}(B_{1}^{1}\otimes L_{1}^{-1}),\cdots, \\
    & R\iota_{n*}(B_{m_{1}-1}^{1}\otimes L_{n}^{1-m_{1}}),\cdots, R\iota_{n*}(B_{n}^{1}\otimes L_{n}^{-1}), \tilde{D}>
  \end{align*}
  such that $\tilde{D}$ is a weak categorical crepant resolution of $Y$.
\end{theorem}
\begin{remark}
  The order of $E_{i}$ in \cref{thm:3.1} does not matter. For any $\sigma$ being permutation of $n$ elements, we can replace $E_{i}$ by $E_{\sigma_{i}}$ and \cref{thm:3.1} still holds.
\end{remark}
\begin{proof}[Proof of \cref{thm:1.11}]
  It follows from \cref{thm:3.1} and \cref{thm:main}.
\end{proof}

\section{A Toy Example}
Now we give a toy example where all the above theorems apply. Given a positive integer $n$, let $Z_{n}:=\mb{A}^{2n}=Spec\ \mb{F}[X_{1},\cdots, x_{n},y_{1},\cdots, y_{n}]$ and $Y_{n}$ be the zero locus of the following equation in $Z_{n}$:
$$f_{n}=x_{1}y_{1}+x_{2}y_{2}+\cdots+x_{n}y_{n}.$$
The singular locus of $Y_{n}$ is the origin $O_{n}:=(0,\cdots, 0)\in \mb{A}^{2n}$. In this section, we take two smooth subvarieties which contain $O_{n}$:
\begin{enumerate}
\item First we take
  $$X_{n}:=\mb{A}^{n}\times \{0\}$$
  which contains elements that $y_{1}=y_{2}=\cdots =y_{n}=0$. In this situation, we apply \cref{thm:1.4} and the zero locus $B=\{O_{n}\}$ is smooth. Thus $Bl_{X_{n}}Y_{n}$ is smooth.
\item Then we take
  $$X_{n}'=\{O_{n}\}.$$
  We apply \cref{thm:main} and in this situation, we have $k=2$. The zero locus of $s$ is
  $$\{[x_{1},\cdots,x_{n},y_{1},\cdots,y_{n}]\in \mb{P}^{2n-1}|f_{n}(x_{1},\cdots,x_{n},y_{1},\cdots ,y_{n})=0\}$$
  which is smooth. Thus $Bl_{X_{n}'}Y_{n}$ is also smooth, too.
\end{enumerate}
By \cref{thm:1.11}, the above two resolutions induce two weakly categorical crepant resolutions of $D^{b}(Y_{n})$, which we denote as $D_{n}$ and $D_{n}'$ respectively. Moreover, we have $D_{1}\cong D_{1}'$ be the theory of Atiyah flops.
\bibliography{Resolution.bib}
\bibliographystyle{plain}
\vspace{5mm}
Kavli Institute for the Physics and 
Mathematics of the Universe (WPI), University of Tokyo,
5-1-5 Kashiwanoha, Kashiwa, 277-8583, Japan.

\textit{E-mail address}: yu.zhao@ipmu.jp
\end{document}